\documentclass[10pt,reqno]{amsart}
\theoremstyle{definition}
\newtheorem{theorem}{Theorem}[section]

\newtheorem{lemma}[theorem]{Lemma}
\newtheorem{corollary}[theorem]{Corollary}
\theoremstyle{definition}
\newtheorem{definition}[theorem]{Definition}
\newtheorem{example}[theorem]{Example}
\theoremstyle{remark}
\newtheorem{remark}[theorem]{Remark}

\newcommand{\norm}[1]{\left\lVert#1\right\rVert}

\usepackage{mathtools}

\newcommand{\mo}[1]{\left|#1\right|}

\allowdisplaybreaks[1]

\numberwithin{equation}{section}

\begin{document}

\title[Spectrum of the Rhaly operator on weighted $c_0$ spaces and operator ideals]{Spectral properties of the Rhaly operator on weighted null sequence spaces and associated operator ideals}


\author{Arnab Patra}
\address{Department of mathematics, Indian Institute of Technology Bhilai, Chhattisgarh, India 492015}
\curraddr{}
\email{arnabp@iitbhilai.ac.in}

\author{Jyoti Rani}
\address{Department of mathematics, Indian Institute of Technology Bhilai, Chhattisgarh, India 492015}
\curraddr{}
\email{jyotir@iitbhilai.ac.in}

\author{Sanjay Kumar Mahto}
\address{Department of mathematics, RNAR College Samastipur, LNMU Darbhanga, India}
\curraddr{}
\email{skmahto0777@gmail.com}



\begin{abstract} 
In this article, a comprehensive study is made on the continuity, compactness, and spectrum of the lower triangular terraced matrix, introduced by H. C. Rhaly, Jr. [Houston J. Math. 15(1): 137-146, 1989], acting on the weighted null sequence spaces with bounded, strictly positive weights. Several spectral subdivisions such as point spectrum, residual spectrum, and continuous spectrum are also discussed. In addition, a new class of operator ideal $\chi_{c_0(r)}^{(s)}$ associated to the Rhaly operator on weighted $c_0$ space is defined using the concept of $s$-number and it is proved that under certain condition, $\chi_{c_0(r)}^{(s)}$ forms a quasi-Banach closed operator ideal. 
\end{abstract}
\maketitle

\section{Introduction}

The linear map $C$ which maps a sequence $\{x_n\}$ of real or complex numbers to its sequence of averages $\{\frac{x_1+\cdots + x_n}{n}\}$ is the well-known discrete Ces\`aro operator. Extensive study can be found in the literature on the continuity and spectrum of $C$ defined over  sequence spaces such as $c_0$, $c$, $\ell_p$, $bv_p \ (1 \leq p < \infty)$, $bv_0$, etc \cite{reade1985spectrum,cesaro_c0_c,curbera,gonzalez,cesaro_bvp,cesaro_bv0}. Recently Albanese et al. studied the spectral properties of Ces\`aro operator defined over the weighted sequence spaces \cite{cesaro_lpw,cesaro_c0w,cesaro_l1w}. Several generalised version of the Ces\`aro operator such as $p-$Ces\`aro operator \cite{rhaly,coskun}, discrete generalized Ces\`aro operator \cite{rhaly_discrete,inequality,oam}, $q-$Ces\`aro operator \cite{qcesaro1,qcesaro2,qcesaro3}, etc. are also studied. For a sequence $a=\{a_n\}$ of real or complex numbers Rhaly \cite{rhaly} introduced the terraced matrix $R_a$ also known as Rhaly matrix where
\[
R_a=
\begin{pmatrix}
a_1 & {0} & {0}&{0} & \cdots \vspace{1.5mm}\\
a_2 & a_2 & {0} &{0} &\cdots\vspace{1.5mm}\\
a_3 & a_3 & a_3 &{0}& \cdots\vspace{1.5mm}\\
\vdots & \vdots & \vdots&\vdots & \ddots
\end{pmatrix}
.\]
The operator represented by the above matrix is known as the Rhaly operator. The Ces\`aro operator $C$ can be obtained by taking $a_n = \frac{1}{n}.$ Rhaly \cite{rhaly} also considered the case $a_n = \frac{1}{n^p}, \ p \in \mathbb{R}$ which is the $p$-cesaro operator. Many researchers investigated the spectrum of Rhaly operators defined over several classical sequence spaces such as $c_0$ \cite{rhaly_c0_1,rhaly_c0_2,rhaly_c0_3}, $\ell_p$ \cite{rhaly_lp}, $bv_0$ \cite{rhaly_bv0}, etc. To the best of our knowledge, no investigation is carried out so far on the spectral properties of Rhaly operators over the weighted sequence spaces. The present work is an attempt in this direction. The aim of this paper is to study the continuity, compactness and spectrum of the Rhaly operator defined over the weighted null sequence space $c_0(r)$ where $r = \{r_n\}$ is the weight vector. Also using the concept of $s$-number sequences, we introduce a new class of operator ideal $\chi_{c_0(r)}^{(s)}$ related to the Rhaly operator $R_a$ acting on weighted $c_0$ spaces. It is demonstrated that under some assumptions on the sequence $a=\{a_n\}$, the class forms an quasi-Banach closed operator ideal. Moreover, an inclusion relation among operator ideals are also proved.


The remainder of this paper is organized as follows. Section 2 contains some notations and some useful results. Section 3 deals with the continuity and compactness and various spectral subdivision of Rhaly operator defined over weighted $c_0$ spaces. In section 4 we introduce a new class of operator ideal $\chi_{c_0(r)}^{(s)}$ and discussed few of its properties.

\section{Preliminaries and Notations}
Throughout the article, all the infinite sequences and matrices are indexed by the set of natural numbers $\mathbb{N}.$ Let $T: X \rightarrow Y$ be a bounded linear operator where $X$ and $Y$ are complex Banach spaces. The range space of $T$ and null space of $T$ are denoted by $R(T)$ and $N(T)$ respectively. The Banach-adjoint of $T$, denoted by $T^*$, is a bounded linear operator $T^*: Y^* \rightarrow X^*$ which is defined by
\begin{equation*}
(T^* \phi)(x) = \phi(Tx) \ \ \ \mbox{for all} \ \phi \in Y^* \ \mbox{and} \ x \in X,
\end{equation*}
where $X^*$ and $Y^*$ are the dual spaces of $X$ and $Y$ respectively. $\mathcal{B}(X,Y)$ and $\mathcal{K}(X,Y)$ denote the space of all bounded linear operators and the ideal of all compact operators form $X$ into $Y$. If $X = Y,$ then $\mathcal{B}(X,Y)$ and $\mathcal{K}(X,Y)$ are denoted by $\mathcal{B}(X)$ and $\mathcal{K}(X)$ respectively. For any $T \in \mathcal{B}(X),$ the operator norm is denoted by $\|T\|$ and the norm $\|.\|_{\infty}$ denotes the supremum norm on the sequence space $c_0$. For any operator $T \in \mathcal{B}(X),$ the resolvent set of $T$ is the set of all complex numbers $\lambda$ for which the operator $T - \lambda I$ has a bounded inverse in $X$ where $I$ is the identity operator in $X.$ The resolvent set of $T$ is denoted by $\rho(T,X).$ The complement of the resolvent set in the complex plane $\mathbb{C}$ is called the spectrum of $T$ and it is denoted by $\sigma(T,X)$. The set of points $\lambda \in \mathbb{C}$ for which $N(T- \lambda I) \neq \{0\}$ is called the point spectrum of $T$ and it is denoted by $\sigma_p(T,X)$. The set of points $\lambda \in \mathbb{C}$ for which $N(T- \lambda I) = \{0\},$ and $\overline{R(T- \lambda I)}= X $ but ${R(T- \lambda I)}\neq X $ is called the continuous spectrum of $T$ and denoted by $\sigma_c(T,X),$ and the set of points $\lambda \in \mathbb{C}$ for which $N(T- \lambda I) = \{0\}$ and $\overline{R(T- \lambda I)}\neq X $ is called the residual spectrum of $T$ and it is denoted by $\sigma_r(T,X)$. The three sets $\sigma_p(T,X),\sigma_c(T,X),\sigma_r(T,X)$ are disjoint and their union is the whole spectrum $\sigma(T,X).$

Let $\mathbb{C}^\mathbb{N}$ denotes the space of all complex sequences. For an infinite positive real sequence $r = \{r_k\},$ the weighted null sequence space $c_0(r)$ is defined as $$c_0(r)=\{\{x_k\}\in \mathbb{C}^{\mathbb{N}} : \lim_{k\to\infty} r_kx_k=0\}$$ equipped with the norm $\|x\|_r=\sup_k |x_k|r_k$. If $D_r$ is the diagonal matrix with $i$-th diagonal entry $r_i$ then, $D_r$ is an isometric isomorphism from $c_0(r)$ to $c_0$ since $x=\{x_k\} \in c_0(r)$ implies $D_rx = \{x_k r_k\} \in c_0$, and $\|x\|_r = \|D_rx\|_{\infty}.$ Hence $\|x\|_r$ defines a norm on $c_0(r)$ and $(c_0(r),\|x\|_r)$ is a Banach space under this norm. The dual of $c_0(r)$ is linearly isometric with the weighted sequence space $\ell_1(r^{-1})$ where $r^{-1} = \{\frac{1}{r_n}\}.$ Also it must be noted that if $\inf_k r_k > 0$ then $c_0(r) = c_0$ and the norms are equivalent. Therefore we are interested in the case $\inf_k r_k = 0.$

%
%
%
%
%
%
%
%

 We record the following lemmas which are useful in this sequel.

\begin{lemma} \cite{wilansky_sequence} \label{bounded_c0}
The matrix $A = (a_{nk})$ gives rise to a bounded linear operator $T \in \mathcal{B}(c_0)$ if and only if the following conditions hold,
\begin{enumerate}
\item[(i)] the rows of $A$ are in $\ell_1$ and their $\ell_1$ norms are bounded,
\item[(ii)] the columns of $A$ are in $c_0$.
\end{enumerate}
The operator norm of $T$ is given by the supremum of the $\ell_1$ norms of the rows.
\end{lemma}

\begin{lemma} \cite[Theorem 1]{icmc} \label{icmc}
For two weight vectors $r=\{r_n\}$, $s=\{s_n\}$ and an infinite matrix $A=(a_{nk})$, the matrix $A\in \mathcal{B}(c_0(r),c_0(s))$ if and only if the following conditions hold,
\begin{enumerate}
\item[(i)]$\sup\limits_{n\in\mathbb{N}} s_n\sum\limits_{k=1} ^{\infty} |\frac{a_{nk}}{r_k}|<\infty,$
\item[(ii)]$\lim\limits_{n\to\infty} s_ka_{nk}=0\ \forall \ k\in\mathbb{N}.$
\end{enumerate}
Then the norm of $A$ is given by $\|A\|_{r,s}=\sup\limits_{n\in\mathbb{N}} s_n\sum\limits_{k=1} ^{\infty} |\frac{a_{nk}}{r_k}|$.
\end{lemma}

\begin{lemma} \cite[p. 59]{goldberg} \label{denserange}
The bounded linear operator $T : X \rightarrow Y$ has dense range if and only if $T^*$ is one to one.
\end{lemma}
%

\section{Boundedness and compactness of $R_a$}
Let $r = \{r_n\}$ and $s = \{s_n\}$ are two bounded, strictly positive weight vectors for the sequence spaces $c_0(r)$ and $c_0(s)$ respectively. For any operator $T:c_0(r)\to c_0(s)$ the operator norm is denoted by $\|T\|_{r,s}$ and if $T$ is defined from $c_0(s)$ to $c_0(s)$ then $\|T\|_{r,s}=\|T\|_{s}$. Following theorem is an immediate application of Lemma \ref{icmc}.

\begin{theorem} \label{bounded}
The Rhaly operator $R_a : c_0(r) \to c_0(s)$ is bounded if and only if
\[\left\lbrace s_n a_n \sum\limits_{k=1}^{n}\frac{1}{r_k}\right\rbrace_{n \in \mathbb{N}} \in \ell_{\infty}, \]
and in this case, $\norm{R_a}_{r,s} = \sup\limits_n s_n \mo{a_n} \sum\limits_{k=1}^{n}\frac{1}{r_k}.$
\end{theorem}

\begin{proof}
It follows from Lemma \ref{icmc} that the Rhaly operator $R_a$ is a bounded linear operator from $c_0(r)$ to $c_0(s)$ if and only if
\[\sup\limits_{n} s_n \mo{a_n} \sum\limits_{k=1}^{n}\frac{1}{r_k} < \infty \ \mbox{and} \ \lim_{n \rightarrow \infty} a_n s_n = 0.\]
If $\sup\limits_{n} s_n \mo{a_n} \sum\limits_{k=1}^{n}\frac{1}{r_k} < \infty$ then there exists a positive real number $M$ such that,
\begin{equation}\label{4.2.1}
 s_n \mo{a_n} \sum\limits_{k=1}^{n}\frac{1}{r_k} < M \ \ \mbox{for all} \ \ n \in \mathbb{N}.
\end{equation}
Since $r=\{r_k\}$ is a bounded sequence of strictly positive real numbers, there exists a real number $M_1 > 0$ such that $\forall n \in \mathbb{N},$
\[\sum\limits_{k=1}^{n}\frac{1}{r_k} >n \frac{1}{M_1}. \]
Finally from \eqref{4.2.1} it follows that
\[0 < s_n \mo{a_n}< \frac{MM_1}{n}\ \ \mbox{for all} \ \ n \in \mathbb{N}, \]
and consequently $\lim_{n \rightarrow \infty} a_n s_n = 0.$ Hence $\sup_{n} s_n \mo{a_n} \sum\limits_{k=1}^{n}\frac{1}{r_k} < \infty$ implies $\lim_{n \rightarrow \infty} a_n s_n =0.$ Also the expression of $\norm{R_a}_{r,s}$ is a direct consequence of Lemma \ref{icmc}.
\end{proof}

\begin{theorem} \label{compact}
The Rhaly operator $R_a : c_0(r) \to c_0(s)$ is compact if and only if
\[\left\lbrace s_n a_n \sum\limits_{k=1}^{n}\frac{1}{r_k}\right\rbrace_{n \in \mathbb{N}} \in c_0. \]
\end{theorem}

\begin{proof} Let us consider the linear operator $T_{r,s}: \mathbb{C}^{\mathbb{N}} \rightarrow  \mathbb{C}^{\mathbb{N}}$ defined as
\[T_{r,s}(x) = \left\lbrace a_ns_n \sum\limits_{k=1}^{n} \frac{x_k}{r_k} \right\rbrace_{n \in \mathbb{N}}.\]
Then $D_s R_a = T_{r,s}D_r$ and consequently $R_a : c_0(r) \to c_0(s)$ is a compact map if and only if $T_{r,s} \in \mathcal{K}(c_0).$
Let $\left\lbrace a_n s_n \sum\limits_{k=1}^{n}\frac{1}{r_k}\right\rbrace_{n \in \mathbb{N}} \in c_0$. Define a sequence of operators $\{T_{r,s}^{(k)}\}_{k \in \mathbb{N}}$ where $T_{r,s}^{(k)}:c_0 \rightarrow c_0 $ and
\[T_{r,s}^{(k)}(x)=\left\lbrace a_1s_1\frac{x_1}{r_1}, a_2s_2\left(\frac{x_1}{r_1}+\frac{x_2}{r_2}\right), \cdots, a_ks_k \sum\limits_{i=1}^{k} \frac{x_i}{r_i},0,0, \cdots \right\rbrace.\]
For each $k\in \mathbb{N}$ the operator $T_{r,s}^{(k)}$ is a finite rank operator. Now
\[
\norm{(T_{r,s} - T_{r,s}^{(k)})x}_\infty = \sup\limits_{n>k} \mo{a_n}s_n \sum\limits_{i=1}^{n} \frac{x_i}{r_i} \leq \norm{x}_\infty  \sup\limits_{n>k} \mo{a_n}s_n \sum\limits_{i=1}^{n} \frac{1}{r_i}.
\]
Hence
\[\norm{T_{r,s}- T_{r,s}^{(k)}} \leq \sup\limits_{n>k} \mo{a_n}s_n \sum\limits_{i=1}^{n}\frac{1}{r_i}.\]
Since $a_ns_n \sum\limits_{i=1}^{n}\frac{1}{r_i} \rightarrow 0$ as $n \rightarrow \infty$, by taking $k \rightarrow \infty$ on both side of the above relation we get
\[\norm{T_{r,s} - T_{r,s}^{(k)}} \rightarrow 0. \]
Hence the sequence of operators $\{T_{r,s}^{(k)}\}$ converges to $T_{r,s}$ under the operator norm, and this implies $T_{r,s}\in \mathcal{K}(c_0)$.

\noindent For the converse part we adopt similar approach as in Proposition 2.2 \cite{cesaro_c0w}. Let $R_a:c_0(r) \rightarrow c_0(s)$ is a compact operator. This implies $T_{r,s} \in \mathcal{K}(c_0).$ This also implies $T_{r,s}(B_{c_0}[0,1])$ is a relatively compact set where $B_{c_0}[0,1]$ denotes the closed unit ball in $c_0.$ Therefore using the result of \cite[p. 15]{compact} there exists a sequence $y=\{y_n\} \in c_0$ such that
\[\mo{(T_{r,s}(x))_n} \leq \mo{y_n} \ \ \mbox{for all} \ \ n \in \mathbb{N}, \ x \in B_{c_0}[0,1].\]
Since $y \in c_0,$ for every $\epsilon > 0$ there exists a $n_0 \in \mathbb{N}$ such that,
\[\mo{(T_{r,s}(x))_n} < \epsilon \ \ \mbox{for all} \ \ n > n_0.\]
Let us consider for each $n \in \mathbb{N}$ the sequence $\{u^{(n)}\}$ with  $\{u_k^{(n)}\}_{k \in \mathbb{N}} \in B_{c_0}[0,1]$ such that for $k \in \{1,2, \cdots, n \},$ $u_k^{(n)} = 1$ and $0$ otherwise. Then
\[\mo{(T_{r,s}(u^{(n)}))_n} = s_n \mo{a_n} \sum\limits_{k=1}^{n}\frac{1}{r_k} < \epsilon \ \ \mbox{for all} \ \ n > n_0.\]
Hence $\left\lbrace a_n s_n \sum\limits_{k=1}^{n}\frac{1}{r_k}\right\rbrace_{n} \in c_0$ and this proves the theorem.
\end{proof}

\begin{remark}
Few important remarks on the above result are mentioned below.
\begin{enumerate}
\item[(a)] If $r_n = s_n = 1$ for all $n \in \mathbb{N}$ then both the weighted sequence spaces $c_0(r)$ and $c_0(s)$ reduce to the classical sequence space $c_0.$ In this setting, Theorem \ref{bounded} yields, the Rhaly operator $R_a:c_0 \rightarrow c_0$ is bounded if and only if $\{na_n\} \in \ell_\infty.$ This result is proved by Leibowitz \cite{leibowitz} in Proposition 3.2.

\item[(b)] Since $|a_n| s_n \sum_{k =1}^{n}\frac{1}{s_k} \geq |a_n|, \ n \in \mathbb{N},$ a necessary condition for $R_a \in \mathcal{B}(c_0(s))$ and $R_a \in \mathcal{K}(c_0(s))$ is that $a_n \in \ell_\infty$ and $a_n \in c_0$ respectively.
\end{enumerate}

\end{remark}

 Let $s=\{s_n\}$ be a decreasing, strictly positive sequence of real numbers. Then
\[s_n \mo{a_n} \sum\limits_{k =1}^{n}\frac{1}{s_k} \leq n\mo{a_n}.\]
Therefore $R_a \in \mathcal{B}(c_0(s))$ if $\{na_n\} \in \ell_\infty$ and $R_a \in \mathcal{K}(c_0(s))$ if $\{na_n\} \in c_0.$ Also let $\{e_n\}$ be a sequence in $c_0(r)$ such that the $n$-th entry is one and all other entries are zero. Then for $n \in \mathbb{N},$
\[\norm{R_a e_n}_s = \sup\limits_{k \geq n} s_k \mo{a_k} \geq a_n s_n = \mo{a_n} \norm{e_n}_s.\]
This gives the following corollary.
\begin{corollary}\label{cor2}
Let $s=\{s_n\}$ be a bounded, decreasing, strictly positive sequence of real numbers. Then $R_a \in \mathcal{B}(c_0(s))$ if $\{na_n\} \in \ell_\infty$ and $R_a \in \mathcal{K}(c_0(s))$ if $\{na_n\} \in c_0.$ Also \[\sup\limits_n \mo{a_n} \leq \norm{R_a}_s \leq \sup\limits_n n \mo{a_n}.\]
\end{corollary}

The following example implies the above corollary is merely a sufficient condition.
\begin{example}
 Let for $n \in \mathbb{N},$ $a_n = \frac{1}{\log (n+1)}$ and $s_n = \frac{1}{2^n}.$ Then $s_n \sum_{k =1}^n \frac{1}{s_k} = \frac{2^{n+1} - 2}{2^n},$ $n \in \mathbb{N}.$ Hence
\[a_n s_n \sum\limits_{k =1}^n \frac{1}{s_k} = \frac{2^{n+1} - 2}{2^n \log(n+1)} \to 0 \mbox{ as } n \to \infty.\]
Here, $\{s_n\}$ is a bounded strictly decreasing sequence and $\{na_n\}$ is unbounded but $R_a \in \mathcal{K}(c_0(s)).$
\end{example}

\section{Spectral properties of $R_a$}

From now onwards, let $\{a_n\}$ is a sequence of positive real numbers such that $\lim_{n \to \infty} na_n $ exists finitely and equal to $\chi \neq 0$ and $S=\{a_n:n \in \mathbb{N}\}$. The following lemma (\cite[Lemma 2.7]{rhaly_bv0}, \cite[Lemma 7]{reade1985spectrum}) is useful in this sequel.

\begin{lemma}\label{eq1}
Let $\lambda \in \mathbb{C} \setminus S$ and $\alpha = \Re(\frac{1}{\lambda}).$ Then the following relation holds
\begin{equation*}
\prod\limits_{k = 1}^n \left| 1 - \frac{a_k}{\lambda} \right| \simeq \frac{1}{n^{\alpha \chi}},
\end{equation*}
where the notation $a_n \simeq b_n$ means the sequences $\left\lbrace \frac{a_n}{b_n} \right\rbrace$ and $\left\lbrace \frac{b_n}{a_n} \right\rbrace$ are bounded and for any $z \in \mathbb{C},$ $\Re(z)$ denotes the real part of $z.$
\end{lemma}

Using the similar concept as \cite[Lemma 7]{reade1985spectrum}, we prove a slightly modified result.

\begin{lemma} \label{lemmaeq1}
For a fixed $m \in \mathbb{N},$ let $\lambda \in \mathbb{C} \setminus \{a_k : k = m+1, m+2, \cdots\}$ and $\alpha = \Re(\frac{1}{\lambda}).$ Then the following relation holds
\begin{equation*}
\prod\limits_{k = m+1}^n \left| 1 - \frac{a_k}{\lambda} \right| \simeq \frac{1}{n^{\alpha \chi}}.
\end{equation*}
\end{lemma}

\begin{proof} 
Let $\frac{1}{\lambda} = \alpha + i \beta$ where $\alpha,$ $\beta \in \mathbb{R}$ and $\lim_{n \to \infty} na_n = \chi$ implies $a_n \simeq \frac{\chi}{n}.$ Imitating the calculations of Lemma 7 in \cite{reade1985spectrum} and using the fact that $e^x \geq 1+x$ for all $x \in \mathbb{R},$ we have the following relation
\begin{equation} \label{eql1}
\prod\limits_{k = m+1}^n \left| 1 - \frac{a_k}{\lambda} \right| \leq \mathcal{O}(1) \exp \sum\limits_{k = m+1}^n \left(-  \frac{\chi}{k} \alpha + \frac{(\alpha^2 + \beta^2)}{2} \frac{\chi^2}{k^2} \right).
\end{equation}
Now we use the following inequality
\[\sum\limits_{k = m+1}^n \frac{1}{k} \geq \int\limits_m^n \frac{1}{x+1} dx = \log (n+1) - \log (m+1) \geq \log n - \log (m+1).\] 
Therefore from (\ref{eql1}) we deduce that
\[ \prod\limits_{k = m+1}^n \left| 1 - \frac{a_k}{\lambda} \right| \leq \frac{\mathcal{O}(1) }{n^{\chi \alpha}}.\]
Also note that
\[\sum\limits_{k = m+1}^n \frac{1}{k} \leq \int\limits_{m-1}^{n-1} \frac{1}{x+1} dx = \log n - \log m.\]
Using the above inequality, in a similar way it can be proved that
\[\prod\limits_{k = m+1}^n \left| 1 - \frac{a_k}{\lambda} \right|^{-1} \leq \mathcal{O}(1) n^{\chi \alpha}.\]
This proves the result.
\end{proof}


Next we derive the point spectrum of $R_a$.

\begin{theorem} \label{point}
Let $s = \{s_n\}$ be a bounded, strictly positive sequence such that $R_a \in \mathcal{B}(c_0(s)).$ Then 
\[\sigma_p(R_a, c_0(s)) = \{ \lambda \in S : \lim\limits_{n \to \infty} a_n s_n n^{{\alpha \chi}} = 0 \}\]
where $\alpha = \Re({\frac{1}{\lambda}}) = \frac{1}{\lambda}.$
\end{theorem}

\begin{proof} Let us consider $R_ax=\lambda x$ for some $\lambda \in \mathbb{C}$ and non-zero $x.$ This gives the system of equations
\begin{equation} \label{eqn2}
\left. \begin{aligned}
	a_1x_1 &=& \lambda x_1\\
	a_2(x_1+x_2) &=& \lambda x_2\\
	&\vdots &\\
	a_k(x_1+x_2+\cdots+x_k) &=& \lambda x_k\\
	& \vdots &
	\end{aligned}
	\right\rbrace.
	\end{equation}
	If $\lambda \notin S$ then from the above equations it follows that $x_k = 0$ for all $k \in \mathbb{N}.$ Hence $\sigma_p(R_a, c_0(s)) \subseteq S.$. Any element $\lambda$ of $S$ belongs to $\sigma_p(R_a, c_0(s))$ if and only if $R_a x = \lambda x$ holds for some non-zero $x = \{x_n\}  \in c_0(s).$ Simplifying $R_a x = \lambda x$ we get,
\begin{equation} \label{eq12}
x_n = \prod\limits_{j=m+1}^{n} \left(  \frac{\lambda a_{j-1}^{-1}}{\lambda a_j^{-1} - 1} \right)x_m, \ n = m+1, m+2, \cdots,
\end{equation}
where $x_m$ is the first non-zero entry of $\{x_n\}.$ Now we rearrange the right hand side of the above equation as follows;
\begin{eqnarray*}
x_n = \prod\limits_{j=m+1}^{n} \left(  \frac{\lambda a_{j-1}^{-1}}{\lambda a_j^{-1} - 1} \right)x_m = \prod\limits_{j=m+1}^{n} \left(  \frac{a_j a_{j-1}^{-1}}{1 - \frac{a_j}{\lambda}} \right)x_m &=& \frac{\prod\limits_{j=m+1}^{n} a_j a_{j-1}^{-1}}{\prod\limits_{j=m+1}^{n} \left(1 - \frac{a_j}{\lambda}\right)}x_m \\
&=& \frac{\frac{a_n}{a_m}}{\prod\limits_{j=m+1}^{n} \left(1 - \frac{a_j}{\lambda}\right)}x_m.
\end{eqnarray*}
From Lemma \ref{lemmaeq1} it can be deduced that $\lambda \in \sigma_p(R_a, c_0(s))$ if and only if
\[\lim\limits_{n \to \infty} a_n s_n n^{\alpha \chi} = 0,\]
where $\alpha = \Re({\frac{1}{\lambda}}) = \frac{1}{\lambda}.$ 

\end{proof}

%

\begin{remark}
Here we mention few important observations related to the point spectrum.
\begin{enumerate}

\item[(i)] Since the weight vector $s = \{s_n\}$ is a bounded sequence, and $a_n \simeq \frac{\chi}{n},$ for any $\lambda \in S$ satisfies $\lambda > \chi$ we have 
\[a_n n^{\alpha \chi} \simeq \chi n^{\alpha \chi - 1}\]
where $\alpha = \Re({\frac{1}{\lambda}}) = \frac{1}{\lambda}.$ This proves that $a_ns_n n^{\alpha \chi} \to 0$ as $n \to \infty$ and consequently $\lambda \in \sigma_p(R_a, c_0(s)).$ In this regard we have the following inclusion relation
\[\{\lambda \in S: \lambda > \chi\} \subseteq \sigma_p(R_a, c_0(s)).\]

\item[(ii)] Let $\lambda \in \sigma_p(R_a, c_0(s)).$ Then for all $\mu \in S$ such that $\lambda \leq \mu$ the following relation holds for all $n \in \mathbb{N}$
\[a_n s_n n^{\frac{\chi}{\mu}} \leq a_n s_n n^{\frac{\chi}{\lambda}}.\]
This shows that $\mu \in \sigma_p(R_a, c_0(s))$ for all $\lambda \leq \mu.$
\item[(iii)] If $a_n = \frac{1}{n}$ and $s_n = 1$ for all $n \in \mathbb{N}$ then $\chi = 1$. In this case $a_n s_n n^{\frac{\chi}{a_k}}$ does not converge to zero for any $k \in \mathbb{N}$ and hence $\sigma_p(R_a, c_0(s)) = \emptyset.$ This reflects the case of the point spectrum of the Ces\`aro matrix over $c_0$ space which is obtained by Reade \cite[Theorem 1]{reade1985spectrum}.
\end{enumerate}
\end{remark}

\begin{theorem}\label{thm4}
Let $s = \{s_n\}$ be a bounded strictly positive sequence such that $R_a \in \mathcal{B}(c_0(s))$. Then the following statements hold
\begin{enumerate}
\item[(i)] $0 \notin \sigma_p(R_a^*, c_0(s)^*),$
\item[(ii)] $S \subseteq \sigma_p(R_a^*, c_0(s)^*),$
\item[(iii)] Let $\lambda \in \mathbb{C}\setminus \overline{S}.$ Then $\lambda \in \sigma_p(R_a^*, c_0(s)^*)$ if and only if
\[\sum\limits_{n = 1}^\infty \frac{1}{s_n n^{\alpha \chi}} < \infty,\]
where $\alpha = \Re(\frac{1}{\lambda}).$
\end{enumerate}
\end{theorem}

\begin{proof} 
(i) Let the relation $R_a^* x = 0$ holds for some $x.$ This implies $\sum_{k = n}^\infty a_kx_k = 0$ for all $n \in \mathbb{N}$ and eventually   we have $x_n = 0$ for all $n \in \mathbb{N}.$ This shows that $0 \notin \sigma_p(R_a^*, c_0(s)^*).$

(ii) Consider the equation $R_a^* x = \lambda x$ for some $\lambda \in \mathbb{C}.$ This gives the following system
\begin{equation}
\left. \begin{aligned}
a_1x_1 + a_2x_2 + a_3x_3 + a_4x_4 + \cdots &=& \lambda x_1\\
a_2x_2 + a_3x_3 + a_4x_4 + \cdots &=& \lambda x_2\\
a_3x_3 + a_4x_4 + \cdots &=& \lambda x_3\\
  &\vdots&
\end{aligned}
  \right\rbrace.
\end{equation}
From the above system it follows that for all $n \in \mathbb{N}$ and $n \geq 2$
\begin{equation}\label{4.2.6}
x_n = \prod\limits_{j=1}^{n-1} \left(1- \frac{a_j}{\lambda} \right) x_1.
\end{equation}
If $\lambda = a_l$ for some $l \in \mathbb{N}$ then from the above equation, it follows
\[x_{l+1} = x_{l+2} = \cdots = 0.\]
Hence $\{x_k\} \in \ell_1(s^{-1}).$ It follows that $S \subseteq \sigma_p(R_a^*, c_0(s)^*).$

(iii) Now let $\lambda\in \mathbb{C}\setminus \overline{S}$. Then $R_a^* x = \lambda x$ implies the equation (\ref{4.2.6}) for $\{x_n\}.$ Then using Lemma \ref{eq1} we get for $n \in \mathbb{N},$
\begin{equation*}
|x_n| = \left|\prod\limits_{j=1}^{n-1} \left(1- \frac{a_j}{\lambda} \right) x_1 \right| \simeq  \frac{1}{n^{\alpha \chi}} \left| x_1 \right|,
\end{equation*}
where $\alpha = \Re(\frac{1}{\lambda}).$ Hence $\lambda \in \sigma_p(R_a^*, c_0(s)^*)$ if and only if $\sum_{n = 1}^\infty \frac{1}{s_n n^{\alpha \chi}} < \infty.$
\end{proof}

\begin{remark} \label{remark1}
The following facts are used to obtain the next lemma. For $\lambda \in \mathbb{C} \setminus \{0\}$ and a fixed $\chi > 0$ the following statements hold
\begin{enumerate}
\item[(i)] $ \left| \lambda - \frac{\chi}{2} \right| < \frac{\chi}{2} \mbox{ if and only if } \ \Re\left({\frac{1}{\lambda}}\right) > \frac{1}{\chi},$
\item[(ii)] $ \left| \lambda - \frac{\chi}{2} \right| = \frac{\chi}{2} \mbox{ if and only if } \ \Re\left({\frac{1}{\lambda}}\right) = \frac{1}{\chi},$
\item[(iii)] $ \left| \lambda - \frac{\chi}{2} \right| > \frac{\chi}{2} \mbox{ if and only if } \ \Re\left({\frac{1}{\lambda}}\right) < \frac{1}{\chi}.$
\end{enumerate}
\end{remark}

\begin{lemma}
Let $s = \{s_n\}$ be a bounded strictly positive sequence such that $R_a \in \mathcal{B}(c_0(s))$. Then
\[\sigma_p(R_a^*, c_0(s)^*) \subseteq \left\lbrace \lambda \in \mathbb{C} : \left| \lambda - \frac{\chi}{2} \right| < \frac{\chi}{2}  \right\rbrace \cup S.\]
\end{lemma}

\begin{proof}
It is already proved that $a_n \in \sigma_p(R_a^*, c_0(s)^*)$ for all $n \in \mathbb{N}.$ Let $\lambda \notin S$ such that $\lambda \in \sigma_p(R_a^*, c_0(s)^*).$ Then from Theorem \ref{thm4} it follows that $\sum_{n = 1}^\infty \frac{1}{s_n n^{\alpha \chi}} < \infty,$ where $\alpha = \Re(\frac{1}{\lambda}).$ Now
\[\frac{1}{s_n n^{\alpha \chi}} \geq ||s||_\infty^{-1}\frac{1}{ n^{\alpha \chi}}.\]
holds for all $n \in \mathbb{N}$. Hence
\[\lambda \in \sigma_p(R_a^*, c_0(s)^*) \Rightarrow \sum\limits_{n = 1}^\infty \frac{1}{ n^{\alpha \chi}} < \infty \Rightarrow \alpha \chi >1,\]
and consequently $\lambda$ satisfies $\left| \lambda - \frac{\chi}{2} \right| < \frac{\chi}{2}.$ This proves the lemma.
\end{proof}

The next result provides an estimation of the spectrum of the operator $R_a.$

\begin{theorem} \label{resolvant}
Let $s=\{s_n\}$ be a bounded, decreasing sequence of positive real numbers such that $R_a \in \mathcal{B}(c_0(s))$, then
\[\sigma(R_a, c_0(s)) \subseteq \left\lbrace \lambda \in \mathbb{C} \setminus \bar{S} : \left|\lambda - \frac{\chi}{2}\right| \leq \frac{\chi}{2}  \right\rbrace \cup \bar{S}.\]

%
\end{theorem}

\begin{proof}
From the inclusion relation $(ii)$ in Theorem \ref{thm4} we have, $S \subseteq \sigma_p(R_a^*, c_0(s)^*)$. Also it is known that for a Banach space $X$, $\sigma_p(T^*, X^*) \subseteq \sigma(T, X).$ Therefore we have $S \subseteq \sigma(R_a, c_0(s))$ with the fact that $\sigma(R_a, c_0(s))$ is a closed subset of $\mathbb{C}.$ Hence $\overline{S} \subset \sigma(R_a, c_0(s))$. Again for $\lambda \notin \overline{S}$ the inverse $(R_a - \lambda I)^{-1}$ exists and has the matrix representation
\begin{equation}
(R_a - \lambda I)^{-1} = (b_{nk}) = \left\lbrace \begin{aligned}
&\frac{1}{a_n - \lambda}, \ \ \ \ \ \ \ \ n = k\\
&\frac{-a_n}{\lambda^2 \prod\limits_{j = k}^n (1 - \frac{a_j}{\lambda})}, \ \ \  1 \leq k < n \\
&0, \ \ \ \ \ \ \ \ \ \ \ \ \mbox{otherwise.}
\end{aligned} \right.
\end{equation}
From Lemma \ref{icmc} it follows that $(R_a - \lambda I)^{-1} = (b_{nk}) \in \mathcal{B}(c_0(s))$ if and only if both the conditions $\sup\limits_{n \in \mathbb{N}} s_n \sum\limits_{k = 1}^{n} \frac{|b_{nk}|}{s_k}<\infty,$ and $\lim\limits_{n \to \infty} s_n |b_{nk}| = 0, \ \forall k \in \mathbb{N}$ are satisfied. Let $\lambda \in \left\lbrace \lambda \in \mathbb{C} \setminus \bar{S} : \left|\lambda - \frac{\chi}{2}\right| > \frac{\chi}{2}  \right\rbrace$. Then we have for $n \in \mathbb{N}$
\begin{eqnarray} \label{eqnew1}
s_n \sum\limits_{k = 1}^{n} \frac{|b_{nk}|}{s_k} &=& s_n \sum\limits_{k = 1}^{n-1} \frac{|b_{nk}|}{s_k} + |b_{nn}|\nonumber \\
												 &=& s_n \sum\limits_{k = 1}^{n-1} \frac{a_n}{\left| \lambda^2 \prod\limits_{j = k}^n (1 - \frac{a_j}{\lambda})\right| s_k} + \left| \frac{1}{a_n - \lambda} \right|.
\end{eqnarray}
Also, since $\lim_{n \to \infty} \frac{n^a}{(n+1)^a} = 1$ for all real $a$, Lemma \ref{lemmaeq1} yields
\begin{equation} \label{eq10}
 \frac{D_1}{(n+1)^{\alpha \chi}} \leq \prod\limits_{k = 1}^n \left| 1 - \frac{a_k}{\lambda} \right| \leq \frac{D_2}{(n+1)^{\alpha \chi}},
\end{equation}
for some positive constants $D_1$ and $D_2$ and Lemma \ref{lemmaeq1} and equation (\ref{eq10}) together imply
\begin{equation} \label{eq11}
M_1 a_n \frac{n^{\alpha \chi}}{k^{\alpha \chi}} \leq \frac{a_n}{\prod\limits_{j=k}^n \left| 1 - \frac{a_j}{\lambda} \right|} \leq M_2 a_n \frac{n^{\alpha \chi}}{k^{\alpha \chi}},
\end{equation}
for some positive constants $M_1$ and $M_2.$ Therefore using the relation (\ref{eq11}) in the equation (\ref{eqnew1}) we obtain for $n \in \mathbb{N}$
\begin{eqnarray*}
s_n \sum\limits_{k = 1}^{n} \frac{|b_{nk}|}{s_k} &\leq & M_2 \frac{s_n}{|\lambda|^2} \sum\limits_{k = 1}^{n-1} \frac{a_n n^{\alpha \chi}}{s_k k^{\alpha \chi}} + \frac{1}{d_\lambda},
\end{eqnarray*}
where $d_\lambda = \mbox{dist } (\lambda, \bar{S}).$ Therefore we have
\[\sup_n a_ns_n n^{\alpha \chi} \sum\limits_{k = 1}^{n-1} \frac{1}{s_k k^{\alpha \chi}} < \infty \Rightarrow \sup\limits_{n} s_n \sum\limits_{k = 1}^{n} \frac{|b_{nk}|}{s_k} < \infty.\]
As  $\lambda \in \left\lbrace \lambda \in \mathbb{C} \setminus \bar{S} : \left|\lambda - \frac{\chi}{2}\right| > \frac{\chi}{2}  \right\rbrace$,
 this is equivalent to $\alpha \chi < 1.$ In this case for $n \in \mathbb{N},$
\begin{eqnarray} \label{eq32}
	a_n n^{\alpha \chi} \sum\limits_{k = 1}^{n-1} \frac{s_n}{s_k k^{\alpha \chi}} &\leq & (na_n) n^{\alpha \chi - 1} \sum\limits_{k = 1}^{n-1} \frac{1}{ k^{\alpha \chi}} \leq D_1 n^{\alpha \chi - 1} \sum\limits_{k = 1}^{n-1} \frac{1}{ k^{\alpha \chi}}
\end{eqnarray}
for some positive constant $D_1$ as $\{na_n\}$ is convergent sequence and $\{s_n\}$ is decreasing. Also
\[ \sum\limits_{k=1}^{n-1} \frac{1}{k^{\alpha \chi}} \leq  \int\limits_0^{n-1} \frac{1}{x^{\alpha \chi}} dx =  \lim\limits_{\epsilon \to 0+} \int\limits_\epsilon^{n-1} \frac{1}{x^{\alpha \chi}} dx = \frac{(n-1)^{1 - \alpha\chi}}{1- \alpha\chi}.\]
The above inequality and (\ref{eq32}) together imply
\[\sup_n a_ns_n n^{\alpha \chi} \sum\limits_{k = 1}^{n-1} \frac{1}{s_k k^{\alpha \chi}} < \infty.\] 
Also from (\ref{eq11}) we have for each $k \in \mathbb{N}$ and large $n$
\[0 \leq s_n |b_{nk}| = \frac{s_na_n}{\left| \lambda^2 \prod\limits_{j = k}^n (1 - \frac{a_j}{\lambda})\right|} \leq M_2 a_n s_n \frac{n^{\alpha \chi}}{|\lambda|^2 k^{\alpha \chi}}.\]
Hence $\lim_{n \to \infty} a_n s_n n^{\alpha \chi} = 0  \Rightarrow \lim_{n \to \infty} s_n |b_{nk}| = 0.$

Again
$0 \leq a_n s_n n^{\alpha \chi} = (na_n) s_n n^{\alpha \chi - 1} \leq D_2 n^{\alpha \chi - 1}$
holds for some positive constant $D_2$ as $\{na_n\}$ and $\{s_n\}$ are bounded. This also proves that $\lim\limits_{n \to \infty} a_n s_n n^{\alpha \chi} = 0.$ Hence $(R_a - \lambda I)^{-1}$ exists and $(R_a - \lambda I)^{-1} \in \mathcal{B}(c_0(s))$ and consequently $\lambda \in \rho(R_a, c_0(s)).$ Finally we have
\[\left\lbrace \lambda \in \mathbb{C} \setminus \bar{S} : \left|\lambda - \frac{\chi}{2}\right| > \frac{\chi}{2}  \right\rbrace \subseteq \rho(R_a, c_0(s)).\]
This proves the result.

\end{proof}


Now we summarize all the result on the spectrum and fine spectrum of the Rhaly operator over the sequence space $c_0(s).$ Consider the following sets
\begin{eqnarray*}
A_1 &=& \left\lbrace \lambda \in S: \lim\limits_{n \to \infty} a_n s_n n^{\alpha \chi} = 0  \right\rbrace, \\
A_2 &=& \left\lbrace \lambda \in \mathbb{C} \setminus (S \cup \{ 0 \}) : \sum\limits_{n =1}^\infty \frac{1}{s_n n^{\alpha \chi}} < \infty  \right\rbrace,
\end{eqnarray*}
where $\alpha = \Re(\frac{1}{\lambda}).$ In this setting we have the following result.

\begin{theorem} \label{main_spectrum}
Let $\{s_n\}$ be a bounded and strictly positive sequence such that $R_a \in \mathcal{B}(c_0(s))$ and $\lim_{n \to \infty} n a_n = \chi \neq 0$ then the following statements hold
\begin{enumerate}
\item[(i)] $\sigma_p(R_a, c_0(s)) = A_1,$
\item[(ii)] $\sigma_p(R_a^*, c_0(s)^*) = A_2 \cup S,$
\item[(iii)] $\sigma_r(R_a, c_0(s)) = (A_2 \cup S) \setminus A_1 = A_2 \cup (S \setminus A_1).$
\end{enumerate}
In addition, if the sequence $\{s_n\}$ is a decreasing sequence then
\begin{enumerate}
\item[(iv)] $\sigma(R_a, c_0(s)) \subseteq \left\lbrace \lambda \in \mathbb{C} \setminus \bar{S} : \left|\lambda - \frac{\chi}{2}\right| \leq \frac{\chi}{2}  \right\rbrace \cup \bar{S},$
\item[(v)] $ \{0\} \subseteq \sigma_c(R_a, c_0(s)) \subseteq (\left\lbrace \lambda \in \mathbb{C} \setminus \bar{S} : \left|\lambda - \frac{\chi}{2}\right| \leq \frac{\chi}{2}  \right\rbrace \cup \bar{S}) \setminus (A_2 \cup S).$
\end{enumerate}
\end{theorem}

\begin{proof}
The relations (i) and (ii) are direct consequences of Theorem \ref{point} and Theorem \ref{thm4} respectively. Also from lemma \ref{denserange} it follows that $\sigma_r(R_a, c_0(s)) = \sigma_p(R_a^*, c_0(s)^*) \setminus \sigma_p(R_a, c_0(s))$ which gives the expression for $\sigma_r(R_a, c_0(s))$ in (iii). The relation (iv) follows from Theorem \ref{resolvant}. The first inclusion relation in (v) follows from Theorem \ref{point} and Theorem \ref{thm4}(i) and the other inclusion relation in (v) follows from the fact that $\sigma_p(R_a, c_0(s)),$ $\sigma_r(R_a, c_0(s)),$ and $\sigma_c(R_a, c_0(s))$ forms a partition of $\sigma(R_a, c_0(s)).$ 
\end{proof}

\section{Operator Ideal}
The theory of operator ideals holds a distinctive place in functional analysis because of its numerous applications in spectrum theory, Banach space geometry, the theory of eigenvalue distributions, etc. Many researchers have developed classes of operator ideals that result from operators acting over sequence spaces \cite{maji2014operator,maji2015some,yaying2021sequence}. In this section, we introduce a new class of operator ideal with the help of $s$-number sequences and discussed of its properties. Let $X$, $Y$, $Z$, $X_0$ and $Y_0$ are the Banach spaces, $\mathcal{B}$ denotes the class of all bounded linear operators between any pair of Banach spaces, $X'$ denotes the dual of $X$, $x'$ denotes the continuous linear functional on $X$. The sequences $r=\{r_n\}$, $t=\{t_n\}$ and $u=\{u_n\}$ are the weight vectors corresponding to sequence spaces $c_0(r)$, $c_0(t)$ and $c_0(u)$ respectively. Let $x' \in X'$ and $y \in Y$ then the mapping $(x'\otimes y) : X \to Y$ is defined by
$(x' \otimes  y)(x) = x'(x)y,~ x \in X$. Before proceeding, we list certain known definitions and results which are necessary for our results.
\begin{definition}\cite[p.79]{pietsch1986eigenvalues}
A map $s: \mathcal{B} \to \mathbb{R}^{\mathbb{N}}$ which assigns every operator $\phi \in \mathcal{B}$ to a non-negative sequence $\{s_n(\phi$)\} is called a $s$-number sequence if it satisfies the following properties: 
\begin{itemize}
	\item [(i)] $\|\phi\|=s_1(\phi)\ge s_2(\phi)\ge \cdots \ge 0$ for $\phi \in \mathcal{B}(X,Y)$,
	\item [(ii)] $s_{m+n-1}(\phi+\psi) \le s_m(\phi)+s_n(\psi)$ for $\phi,\psi \in \mathcal{B}(X,Y)$,
	\item [(iii)] $s_n(\zeta \phi \eta)\le \|\zeta\| s_n(\phi) \|\eta\|$ for $\zeta\in \mathcal{B}(Y,Y_0)$, $\phi \in \mathcal{B}(X,Y)$, $\eta \in \mathcal{B}(X_0,X)$,
	\item [(iv)] If rank$(\phi)<n$, then $s_n(\phi)=0$,
	\item [(v)] $s_n(I_2:\ell_2^{(i)} \to \ell_2^{(i)})=1$, where $I_2$ denotes the identity operator on the $i$-dimensional Hilbert space $\ell_2^{(i)}$.
\end{itemize}
We call $s_n(\phi)$ the $n$-th $s$-number of the operator $\phi$.
\end{definition}
\begin{definition}\cite[p.25]{pietsch1986eigenvalues}
For a subset $\mathcal{A}(X,Y)$ of $\mathcal{B}(X,Y)$, the class $\mathcal{A}=\cup_{X,Y}\mathcal{A}(X,Y)$ is said to be an operator ideal if each component $\mathcal{A}(X,Y)$ satisfies the following conditions:
\begin{itemize}
	\item [(i)] $x'\otimes y \in \mathcal{A}(X,Y)$ for $x' \in X'$ and $y \in Y$,
	\item [(ii)] $\phi +\psi \in \mathcal{A}(X,Y)$ for $\phi, \psi \in \mathcal{A}(X,Y)$,
	\item [(iii)] $\zeta \phi \eta\in \mathcal{A}(X_0, Y_0)$ for $\eta \in \mathcal{B}(X_0,X)$, $\phi \in \mathcal{A}(X,Y) $, $\zeta \in \mathcal{B}(Y,Y_0)$.
\end{itemize}
\end{definition}
\begin{definition}\cite[p.25]{pietsch1986eigenvalues}
	A function $\Gamma : \mathcal{A} \to \mathbb{R}^+$ which assigns to every operator $\phi \in \mathcal{A}$ a non-negative number $\Gamma(\phi)$ is called a quasi-norm on the operator ideal $\mathcal{A}$ if it satisfies the following properties:
	\begin{itemize}
		\item [(i)] $\Gamma(x' \otimes y)=\|x'\| \|y\|$ for $x' \in X'$ and $y \in Y$,
		\item [(ii)] $\Gamma(\phi+\psi) \le c(\Gamma(\phi)+\Gamma(\psi))$ for $\phi, \psi \in \mathcal{A}(X,Y)$ where $c \ge 1$ is a constant,
		\item [(iii)] $\Gamma(\zeta \phi \eta) \le \|\zeta\| \Gamma(\phi) \|\eta\|$ for $\eta \in \mathcal{B}(X_0,X), \phi\in \mathcal{A}(X,Y), \zeta \in \mathcal{B}(Y,Y_0)$.
	\end{itemize}
\end{definition}
\begin{definition}\cite[p.81]{pietsch1986eigenvalues}
An s-number $s$ is called multiplicative if
$$s_{m+n-1}(\phi\psi)\le s_m(\phi)s_n(\psi) ~~\text{for}~~\psi\in \mathcal{L}(X,Y) ~~ \text{and}~~\phi\in \mathcal{L}(Y,Z).$$
\end{definition}
\begin{definition}\cite[p.26]{pietsch1986eigenvalues}
An operator ideal $\mathcal{A}$ is closed if all components $\mathcal{A}(X,Y)$  are closed linear subsets of $\mathcal{B}(X,Y)$.
\end{definition}
The following lemma is useful in this context.
\begin{lemma}\label{lemma12}\cite[p.80]{pietsch1986eigenvalues}
If $\phi, \psi \in \mathcal{B}(X,Y)$ then $|s_i(\phi)-s_i(\psi)| \le \| \phi- \psi\|, ~~i \in \mathbb{N}.$
\end{lemma}
Now we define sequence space $\chi_{c_0(r)}$ associated to the Rhaly operator $R_a$ as follows 
\[\chi_{c_0(r)}=\{v\in \mathbb{C}^{\mathbb{N}}:R_a(v)\in c_{0}(r)\}.\]
Thus \[\chi_{c_0(r)}=\left\lbrace v\in \mathbb{C}^{\mathbb{N}} :\lim_{i \to \infty}\left( \sum_{j=1}^ia_iv_j\right) r_i =0\right\rbrace .\]
An operator $\phi \in \mathcal{B}(X,Y)$ is called s-type $\chi_{c_0(r)}$ operator if
\[ \lim_{i \to \infty}\left( a_i \sum_{j=1}^is_j(\phi)\right) r_i = 0\]
We denote the class of all s-type $\chi_{c_0(r)}$ operators by $\chi_{c_0(r)}^{(s)}$ i.e.,
\[\chi_{c_0(r)}^{(s)}=\left\lbrace \phi\in \mathcal{B}:  \lim_{i \to \infty}\left( a_i \sum_{j=1}^is_j(\phi)\right) r_i = 0\right\rbrace .\]
If $\phi \in \mathcal{B}(X,Y)$ then 
\[\chi_{c_0(r)}^{(s)}(X,Y)=\left\lbrace \phi\in \mathcal{B}(X,Y):  \lim_{i \to \infty}\left( a_i \sum_{j=1}^is_j(\phi)\right) r_i = 0\right\rbrace .\]

\begin{lemma} 
If $r_n\le t_n$ for all $n$ then $\chi_{c_0(t)}^{(s)} \subseteq \chi_{c_0(r)}^{(s)}$.
\end{lemma}
\begin{proof}
	Suppose $\phi \in \chi_{c_0(t)}^{(s)}$ then 
	\[\lim_{i \to \infty}\left(a_i \sum_{j=1}^{i}s_j(\phi)\right)t_i=0  \]
	Now, 
	\begin{align*}
	\lim_{i \to \infty}\left(a_i \sum_{j=1}^{i}s_j(\phi)\right)r_i \le \lim_{i \to \infty}\left(a_i \sum_{j=1}^{i}s_j(\phi)\right)t_i
	=0
	\end{align*}
Thus $\phi \in \chi_{c_0(r)}^{(s)}$.
\end{proof}

\begin{theorem}
The following statements hold,
	\begin{itemize}
		\item [(i)]$\chi_{c_0(r)}^{(s)}$ is an operator ideal if $\{a_n\} \in c_0(r)$,
		\item [(ii)]$\chi_{c_0(r)}^{(s)}$ is a closed operator ideal if $\{na_n\} \in c_0(r)$.
	\end{itemize}

\end{theorem}
\begin{proof}(i)
	Let $X$ and $Y$ are any two Banach spaces.
Let $x' \in X'$ and $y \in Y$. Then $x'\otimes y$ is a rank one operator so $s_i(x'\otimes y)=0$ for all $i \ge 2$. Consider,
\begin{align*}
\lim_{i \to \infty}\left( a_i \sum_{j=1}^is_j(x' \otimes y)\right)r_i=& \lim_{i \to \infty }a_is_1(x' \otimes y)r_i \\
=& \|x' \otimes y \| \lim_{i \to \infty}a_ir_i	\\
=& 0.
\end{align*}	
Thus, $(x' \otimes y)\in \chi_{c_0(r)}^{(s)}(X,Y). $\\
Suppose $\phi, \psi \in \chi_{c_0(r)}^{(s)}(X,Y)$. By using properties of $s$-number, we have 
\begin{align*}
\lim_{ i \to\infty} a_i \left( \sum_{j=1}^i s_j(\phi +\psi )\right) r_i
=&\lim_{ i \to\infty} \left( a_i\left( \sum_{j=1}^is_{2j-1}(\phi+ \psi)+ \sum_{j=1}^is_{2j}(\phi+ \psi)\right) \right)  r_i \\
\le & \lim_{ i \to\infty} \left( a_i\left( \sum_{j=1}^is_{2j-1}(\phi+ \psi)+ \sum_{j=1}^is_{2j-1}(\phi+ \psi)\right) \right)  r_i \\ 
=&2\lim_{l \to \infty}a_i\left( \sum_{j=1}^{i} s_{2j-1}(\phi+\psi)\right) r_i\\
\le & 2a_i \lim_{i \to \infty}\left( \sum_{j=1}^{i}\left( s_j(\phi)+s_j(\psi)\right) \right) r_i \\
=&2 \lim_{i \to \infty}a_i\left( \sum_{j=1}^{i}s_j(\phi)\right) r_i+ 2 \lim_{i \to \infty}a_ia_n\left( \sum_{j=1}^{i}s_j(\psi)\right) r_i \\
=& 0.
\end{align*}
This implies, $\lim_{ i \to\infty} a_i \left( \sum_{j=1}^i s_j(\phi +\psi )\right) r_i=0$. Thus, $\phi +\psi \in \chi_{c_0(r)}^{(s)}(X,Y) $. \\
Suppose  $\zeta \in \mathcal{B}(Y,Y_0),\eta \in \mathcal{B}(X_0,X)$ and $ \phi \in \chi_{c_0(r)}^{(s)}(X,Y)$. Then
\begin{align*}
	\lim_{i \to \infty}a_i\left( \sum_{j=1}^{i}s_j(\zeta \phi \eta)\right) r_i =\|\zeta\|\|\eta\|\lim_{i \to \infty}a_i\left( \sum_{j=1}^{i}s_j(\phi)\right) r_i 
	= 0 .
\end{align*}
Thus, $\zeta \phi \eta \in \chi_{c_0(r)}^{(s)}(X_0,Y_0) $. Hence $\chi_{c_0(r)}^{(s)}$ is an operator ideal.\\
(ii) Let $\{\phi_k\}$ be any sequence in $\chi_{c_0(r)}^{(s)}(X,Y)$ converging to $\phi \in \mathcal{B}(X,Y)$ with respect to the operator norm. Given $\epsilon>0$, we can fix $k_0 \in \mathbb{N}$ such that 
$$\|\phi-\phi_k\|\le \epsilon ~~~~\text{for}~ k \ge k_0.$$
Now, 
\begin{align*}
	 \lim_{i \to \infty}\left( a_i \sum_{j=1}^is_j(\phi)\right) r_i =& \lim_{i \to \infty}\left( a_i \sum_{j=1}^is_j(\phi-\phi_k+\phi_k)\right) r_i \\
	 \le&  \lim_{i \to \infty}\left( a_i \sum_{j=1}^is_1(\phi-\phi_k)\right) r_i + \lim_{i \to \infty}\left( a_i \sum_{j=1}^is_j(\phi_k)\right) r_i\\
	 =&\|\phi-\phi_k\|\lim_{i \to \infty}ia_ir_i+ \lim_{i \to \infty}\left( a_i \sum_{j=1}^is_j(\phi_k)\right) r_i \\
	 =&0.
\end{align*}
Hence, $\phi \in \chi_{c_0(r)}^{(s)}(X,Y)$ as $\lim_{i \to \infty}\left( a_i \sum_{j=1}^is_j(\phi)\right) r_i =0$.
The required result is proved.

\end{proof}
Define a mapping $Q^{(s)}:\chi_{c_0(r)}^{(s)} \to \mathbb{R}^+$ where $\mathbb{R}^+$ is the set of all positive real numbers, by
\[ Q^{(s)}(\phi)= \sup_i\left| a_i \sum_{j=1}^i s_j(\phi)\right| r_i~~ \text{where $\phi \in\chi_{c_0(r)}^{(s)}$}.\] 
The next result proves that under certain condition $Q^{(s)}$ forms a quasi norm on $\chi_{c_0(r)}^{(s)}$.
\begin{theorem}
	If $\sup_i|a_i|r_i=1$ then
	the mapping $Q^{(s)}$ is a quasi norm on the set $\chi_{c_0(r)}^{(s)}$ and the operator ideal $\chi_{c_0(r)}^{(s)}$ is complete under the quasi-norm $Q^{(s)}$. 
\end{theorem}
\begin{proof}
Let $X$ and $Y$ are any two Banach spaces and $x' \in X'$,$y \in Y$. Now 
\begin{align*}
	Q^{(s)}(x' \otimes y)=& \sup_i\left| a_i \sum_{j=1}^i s_j(x' \otimes y)\right| r_i\\
	=& \sup_i\left| a_i s_1(x' \otimes y)\right| r_i\\
	=& \|x' \otimes y\| \sup_ia_ir_i\\
	=&\|x' \otimes y\|.
\end{align*}
Clearly $\|x' \otimes y \|=\|x'\| \|y\|$, so $Q^{(s)}(x' \otimes y)=\|x'\| \|y\|$.\\
Suppose $\phi, \psi \in \chi_{c_0(r)}^{(s)}(X,Y)$ and consider
\begin{align*}
	 Q^{(s)}(\phi+\psi)= &\sup_i\left| a_i \sum_{j=1}^i s_j(\phi+\psi)\right| r_i\\
	 \le & 2 \sup_i \left| a_i \sum_{j=1}^i(s_j(\phi)+s_j(\psi))\right| r_i\\
	 =& 2 \sup_i\left| a_i\left( \sum_{j=1}^i s_j(\phi)\right)r_i + a_i\left( \sum_{j=1}^i s_j(\psi)\right)r_i\right| \\
	 \le &2 \sup_i\left| a_i\sum_{j=1}^i s_j(\phi)\right| r_i +2 \sup_i \left| a_i \sum_{j=1}^i s_j(\psi)\right|r_i \\
	 \le & 2\left( Q^{(s)}(\phi)+Q^{(s)}(\psi)\right) .	 
\end{align*}
Thus, \[Q^{(s)}(\phi+\psi) \le 2\left( Q^{(s)}(\phi)+Q^{(s)}(\psi)\right) .\]

Let $\phi \in \chi_{c_0(r)}^{(s)}(X,Y) $ and $\zeta \in \mathcal{B}(Y,Y_0)$, $\eta\in \mathcal{B}(X_0,X)$.
\begin{align*}
Q^{(s)}(\zeta \phi \eta)=&\sup_i\left| a_i \sum_{j=1}^i s_j(\zeta \phi \eta)\right| r_i\\
\le& \|\zeta\|\|\eta\|\sup_{i}\left| a_i\left( \sum_{j=1}^{i}s_j(\phi)\right)\right|  r_i. 
\end{align*}
This implies, \[Q^{(s)}(\zeta \phi \eta)\le \| \zeta\| Q^{(s)}(\phi)\|\eta\|. \]
Hence, $Q^{(s)}$ is quasi-norm on operator ideal $\chi_{c_0(r)}^{(s)}$.\\
Also, suppose $\phi\in \chi_{c_0(r)}^{(s)}(X,Y)$. We have
\begin{align*}
	Q^{(s)}(\phi)=&\sup_i\left| a_i \sum_{j=1}^is_j(\phi)\right| r_i\\
	\ge& \sup_i s_1(\phi)a_ir_i\\
	=& \| \phi\| \sup_i a_i r_i\\
	=& \|\phi \|.
\end{align*}
This implies, 
\begin{equation}\label{quasi10}
	\| \phi\| \le 	Q^{(s)}(\phi) \text{~for $\phi \in \chi_{c_0(r)}^{(s)}(X , Y). $}
\end{equation}
\end{proof}
Let $\{\phi_n\}$ be a Cauchy sequence in $\chi_{c_0(r)}^{(s)}(X , Y)$. Then for $\epsilon>0$ there exists $k \in \mathbb{N}$ such that 
\begin{equation}\label{quasi11}
	Q^{(s)}(\phi_n-\phi_m)< \epsilon ~~\text{for all}~~ n,m \ge k.
\end{equation}
By using equation (\ref{quasi10}), we have
 \[\| \phi_n- \phi_m \| < \epsilon ~~\text{for all}~~ n,m \ge k. \]
This implies $\{\phi_n\}$ is a Cauchy sequence in $\mathcal{B}(X,Y)$.
Since $\mathcal{B}(X,Y)$ is a Banach space, there exists $\phi\in \mathcal{B}(X,Y)$ such that $\phi_n \to \phi$ as $n \to \infty$. Using Lemma \ref{lemma12}, we have 
\[ \left|s_i(\phi_n-\phi_m)-s_i(\phi-\phi_m)\right| \le \| \phi_n- \phi\|,~~ i \in \mathbb{N}. \] 
As $\phi_n \to \phi$, taking $n \to \infty$ and keeping $m$ fixed we have
\begin{equation}\label{quasi12}
	s_i(\phi_n-\phi_m) \to s_i(\phi-\phi_m),~~ i \in \mathbb{N}. 
\end{equation}
 Now, we shall show that $\phi_n \to \phi$ as $n \to \infty$ in  $\chi_{c_0(r)}^{(s)}(X,Y)$. Again, from equation (\ref{quasi11}), we have
\[ \sup_i \left| a_i \sum_{j=1}^i s_j(\phi_n-\phi_m)\right|r_i<\epsilon~~\text{for all}~~m \ge k. \]
Keeping $m$ fixed and letting $n \to \infty$, we obtain on using equation (\ref{quasi12})
\begin{equation}\label{quasi13}
	\sup_i \left| a_i \sum_{j=1}^i s_j(\phi-\phi_m)\right|r_i<\epsilon~~\text{for all}~~n,m \ge k.
\end{equation}
which gives $Q^{(s)}(\phi-\phi_m)<\epsilon$ for all $m \ge k$. This implies $\phi_m\to \phi$ under the quasi norm $Q^{(s)}$. Again,
\begin{align*}
\sum_{j=1}^is_j(\phi)=\sum_{j=1}^is_{2j-1}(\phi)+\sum_{j=1}^is_{2j}(\phi)
\le & 2\sum_{j=1}^is_{2j-1}(\phi)\\
=& 2\sum_{j=1}^is_{2j-1}(\phi-\phi_m+\phi_m)\\
\le & 2 \left( \sum_{j=1}^is_j(\phi-\phi_m)+\sum_{j=1}^is_j(\phi_m)\right). 
\end{align*}
Thus,
\begin{equation}\label{quasi14}
	\sum_{j=1}^is_j(\phi) \le 2 \left( \sum_{j=1}^is_j(\phi-\phi_m)+\sum_{j=1}^is_j(\phi_m)\right).
\end{equation}
The above inequality implies that
\begin{align*}
\left|a_i \sum_{j=1}^is_j(\phi)\right| r_i\le&2\left| a_i \left( \sum_{j=1}^is_j(\phi-\phi_m)+\sum_{j=1}^is_j(\phi_m) \right) \right| r_i\\
\le & 2\left| a_i\sum_{j=1}^is_j(\phi-\phi_m)\right|r_i +2\left| a_i\sum_{j=1}^is_j(\phi_m)\right| r_i.
\end{align*}
By using relation (\ref{quasi13}), we obtain that
\[\left|a_i \sum_{j=1}^is_j(\phi)\right| r_i \le 2 \epsilon+2\left|a_i \sum_{j=1}^is_j(\phi_m)\right| r_i~~ \forall m \ge k.\] 
Taking limit on both sides, we obtain that
\[ \lim_{i \to \infty}\left|a_i \sum_{j=1}^is_j(\phi)\right| r_i \le 2\epsilon +2 \lim_{ i \to\infty} \left|a_i \sum_{j=1}^is_j(\phi_m)\right| r_i\]
Then, 
\[\lim_{i \to \infty}\left|a_i \sum_{j=1}^is_j(\phi)\right| r_i \le 2 \epsilon.\]
Thus, $\phi \in \chi_{c_0(r)}^{(s)}(X,Y)$, and
this implies, $\chi_{c_0(r)}^{(s)}$ is complete under quasi-norm $Q^{(s)}$.


\bibliographystyle{plain}
\bibliography{References}
\end{document}